\documentclass[11pt]{article}

\usepackage[margin=1.3in]{geometry}
\usepackage{amssymb,amsmath}
\usepackage{amsthm}
\usepackage{hyperref}
\usepackage{mathrsfs}
\usepackage{textcomp}
\hypersetup{
    colorlinks,%
    citecolor=black,%
    filecolor=black,%
    linkcolor=black,%
    urlcolor=black
}

\newtheorem{thm}{Theorem}[section]
\newtheorem{lem}[thm]{Lemma}
\newtheorem{prop}[thm]{Proposition}

\newtheorem{cor}[thm]{Corollary}

\def \a {\alpha}
\def \e {\epsilon}
\def \d {\delta}

\def\XXint#1#2#3{{\setbox0=\hbox{$#1{#2#3}{\int}$}
  \vcenter{\hbox{$#2#3$}}\kern-.5\wd0}}

\newcommand{\al}{\alpha}                \newcommand{\lda}{\lambda}
                
\newcommand{\va}{\varepsilon}           
\newcommand{\be}{\begin{equation}}      \newcommand{\ee}{\end{equation}}
                 
\newcommand{\Lda}{\Lambda}              
\newcommand{\R}{\mathbb{R}}              

\begin{document}

\title{\textbf{Schauder estimates for an integro-differential equation with applications to a nonlocal Burgers equation}
\bigskip}

\author{Cyril Imbert,\ \ Tianling Jin,\ \ Roman Shvydkoy}

\date{\today}

\maketitle

\begin{abstract}
We obtain Schauder estimates for a general class of linear integro-differential equations. The estimates are applied to a scalar non-local Burgers equation and complete the global well-posedness results obtained in \cite{ISV}.
\end{abstract}

\section{Introduction}
This note studies the classical Schauder estimates for a general class of linear integro-differential equations of the form
\begin{equation}\label{eq:main}
w_t(t,x)= p.v. \int_{\R^n}(w(t,y)-w(t,x))\frac{m(t,x,y)}{|x-y|^{n+1}}dy.
\end{equation}
We assume that $m \in C^{\a}((-6,0]\times \R^n \times \R^n)$, $\lambda \leq m \leq \Lambda$, and $w \in C^{1+\a}((-6,0]\times \R^n)$, for some $\a >0$, $\lambda,\Lambda>0$. Written in terms of $w$-increments, it becomes
\begin{equation}\label{eq:main2}
w_t(t,x)= p.v. \int_{\R^n}(w(t,x+y)-w(t,x))\frac{m(t,x,x+y)}{|y|^{n+1}}dy.
\end{equation}
The kernel $K(t,x,y) = \frac{m(t,x,x+y)}{|y|^{n+1}}$ is typically assumed to satisfy an evenness condition such as $K(t,x,y) = K(t,x,-y)$ which appears naturally in the case when  equation \eqref{eq:main2} represents the generator of a L\'evy process with jumps. We do not make any such assumption. Our motivation primarily comes from studying variational or hydrodynamical models, in particular the non-local Burgers model introduced recently in the works of Lelievre \cite{Lelievre-PHD,Lelievre-art1} for viscous case and developed in the inviscid case by Imbert, Shvydkoy and Vigneron in \cite{ISV}. In this model, the classical Euler equation of conservation of momentum $u_t + u \cdot \nabla u = - \nabla p$ is replaced by a non-local variant
\begin{equation}\label{burgers}
u_t-u(-\Delta)^{\frac 12}u+(-\Delta)^{\frac 12}u^2=0\quad\mbox{in } [0,\infty)\times\R^n.
\end{equation} 
In its integral form, the equation reads
\begin{equation}\label{eq:main3}
u_t(t,x)= p.v. \int_{\R^n}(u(t,y)-u(t,x))\frac{u(t,y)}{|x-y|^{n+1}}dy,
\end{equation}
or for the new variable $w = u^2$ it takes the form of \eqref{eq:main} with 
\begin{equation}\label{eq:m}
m(t,x,y)=C(n)\frac{\sqrt{w(t,x)w(t,y)}}{\sqrt{w(t,x)}+\sqrt{w(t,y)}},
\end{equation}
where $C(n)$ is a positive dimensional constant. 

The emerged symmetry $m(t,x,y) = m(t,y,x)$ in \eqref{eq:m} allows to apply the De Giorgi regularization result in Caffarelli-Chan-Vasseur \cite{CCV} or Felsinger-Kassmann
\cite{FK},  and obtain $C^\a$ bound in space-time in terms of $L^\infty$-norm of the initial condition (note the maximum principle).  Parallel to this, the regularity theory of fully nonlinear integro-differential equations in non-divergence form was developed before in Caffarelli-Silvestre \cite{CS09}, and Lara-D{{\'a}}vila \cite{LD}. 

This implies the $C^{\a}$-regularity of $m$ for solutions bounded away from zero. However, the lack of evenness as stated above makes the equation out of the range of immediate applicability of recently obtained Schauder estimates for similar equations in non-divergence form, such as Mikulevicius-Pragarauskas \cite{MP}, Jin-Xiong \cite{JX} or most recently for the fully non-linear case by Dong-Zhang \cite{DZ}.  It therefore needs to be addressed separately to fulfill the need for higher order regularity which should come naturally from parabolic nature of the equation.

On the first step, we partially restore the evenness by ``freezing the coefficients" and introducing an active source term, i.e. rewriting \eqref{eq:main} as 
\begin{equation}\label{eq:main4}
\begin{split}
w_t(t,x)&= \int_{\R^n}(w(t,y)-w(t,x)- \nabla w(t,x) \cdot (y-x) \chi_{|y-x|\le 1})\frac{m(t,x,x)}{|x-y|^{n+1}}dy\\
&+ \int_{\R^n}(w(t,y)-w(t,x))\frac{m(t,x,y)-m(t,x,x)}{|x-y|^{n+1}}dy.
\end{split}
\end{equation}
In this form it is clear that the $C^{1+\a}$ regularity of $w$ and $C^{\a}$ regularity of $m$ are sufficient to make sense of both integrals classically. Moreover, the gradient term in the first one is superfluous due to vanishing, and thus not changing the equation. We therefore will take a more general approach and study a slightly broader class of equations, namely
\begin{equation}\label{eq:divergence form}
\begin{split}
u_t(t,x)&=\int_{\R^n}(u(t,x+y)-u(t,x))K(t,x,y)dy\\
&+\int_{\R^n}(u(t,x+y)-u(t,x))G(t,x,y)dy+f(t,x)\quad\mbox{in } (-6,0]\times\R^n,
\end{split}
\end{equation}
where $K$ and $G$ satisfy 
\begin{itemize}
\item[\textbf{(K1)}]
  $K(t,x,y)=K(t,x,-y)\quad \mbox{for all } (t,x,y)\in
  (-6,0]\times\R^n\times\R^n,$

\item[\textbf{(K2)}]
  $\lda|y|^{-n-1}\le K(t,x,y)\le \Lda |y|^{-n-1}\quad \mbox{for all }
  (t,x,y)\in (-6,0]\times\R^n\times\R^n,$

\item[\textbf{(K3)}]
  $|K(t_1,x_1,y)-K(t_2,x_2,y)|\le \Lambda
  (|x_1-x_2|^{\alpha} +|t_1-t_2|^{\al})|y|^{-n-1}$
  for all  $(t_1,x_1,y)$, $(t_2,x_2,y)\in (-6,0]\times\R^n\times\R^n.$

\item[\textbf{(G1)}]
  $|G(t,x,y)|\le \Lda \min(1,|y|^\alpha)|y|^{-n-1}\quad \mbox{for all
  } (t,x,y)\in (-6,0]\times\R^n\times \R^n,$
\item[\textbf{(G2)}]
  $|G(t_1, x_1,y)-G(t_2,x_2,y)| \le \Lambda \min(|x_1-x_2|^{\alpha}
  +|t_1-t_2|^{\al},|y|^\alpha)|y|^{-n-1}$
  for all $(t_1,x_1,y)$, $(t_2,x_2,y)\in (-6,0]\times\R^n\times\R^n.$

\end{itemize}
Note that $K$ is  assumed to be even in $y$, but $G$ is \emph{not} assumed to be even in $y$, and all the assumptions are satisfied if $K$ and $G$ are derived  from \eqref{eq:main4}. We assume that $f$ is a passive source term independent of the solution. 
In such formulation of the original equation we can view \eqref{eq:divergence form} as a perturbation of the symmetric case and use \cite{DZ,JX} to obtain the higher
order regularity estimates for \eqref{eq:divergence form} and hence for \eqref{eq:main}.

\begin{thm}\label{thm:schauder-div}
  Suppose $u\in C^{1+\a}((-6,0])\times\R^n)$ is a solution of \eqref{eq:divergence form} with $f \in C^\alpha_{x,t} ((-6,0])\times\R^n)$. Suppose $K$ and
  $G$ satisfy {\rm (K1), (K2), (K3), (G1), (G2)}. Then for
  every $\beta<\alpha$, there exists $C>0$ depending only on
  $n,\lda,\Lda,\alpha,\beta$ such that
\begin{equation}
\|u\|_{C^{1+\beta}_{x,t}((-1,0])\times\R^n)}\le C (\|u\|_{L^\infty((-6,0])\times\R^n)}+\|f\|_{C^\beta_{x,t}((-6,0])\times\R^n)}).
\end{equation}
\end{thm}
At the end of this note we will elaborate more on how to obtain global in time periodic solutions to \eqref{burgers} with the help of Theorem~\ref{thm:schauder-div} (see also the discussion in  \cite{ISV}). With regard to the higher order regularity the relation \eqref{eq:m} clearly allows to bootstrap on the gain of smoothness. We obtain the following as  a consequence.
\begin{cor}\label{cor:burgers}
  Suppose $u$ is a positive smooth periodic (in $x$) solution of
  \eqref{burgers} in $(-6,0]\times\R^n$. Then for every
  positive integer $k$, there exists $C>0$ depending only on
  $n,k,\|u\|_{L^\infty((-6,0]\times\R^n)}$ and
  $\min_{(-6,0]\times\R^n}u$ such that
\[
\|u\|_{C^k_{x,t}((-1,0]\times\R^n)}\le C.
\]
\end{cor}

\bigskip

\noindent\textbf{Acknowledgements:}  The work of R.S. is partially supported by NSF grants DMS-1210896 and DMS-1515705. The authors thank  Luis Silvestre  for motivating and fruitful discussions.

\section{Preliminary}

We first deal with the symmetric case ($G=0$). Suppose that $f(t,x)\in C^\alpha((-6,0]\times\R^n)$, and $u\in C^{1+\alpha}((-6,0]\times\R^n)$ is a solution of
\begin{equation}\label{eq:aux}
u_t(t,x)=\int_{\R^n}(u(t,x+y)-u(t,x))K(t,x,y)dy+f(t,x) \quad\mbox{in }(-6,0]\times\R^n.
\end{equation}
\begin{prop}\label{thm:schauder in xt}
Suppose $K$ satisfies  {\rm (K1), (K2), (K3)}. There exists $C>0$ depending only on $n,\lambda,\Lambda$ such that
\[
\|\nabla_x u\|_{C^\alpha((-2,0]\times\R^n)}+\|u_t\|_{C^\alpha((-2,0]\times\R^n)}\le C (\|u\|_{L^\infty((-6,0])\times\R^n)}+\|f\|_{C^\alpha((-6,0]\times\R^n)}).
\]
\end{prop}
\begin{proof}
First of all, we know from the H\"older estimates in \cite{LD} that there exist $C,\gamma>0$ depending only on $n,\lambda,\Lambda$ such that
\[
\|u\|_{C^\gamma((-5,0]\times\R^n)}\le C (\|u\|_{L^\infty((-6,0])\times\R^n)}+\|f\|_{L^\infty((-6,0]\times\R^n)}).
\]
Then it follows from Theorem 1.1 in \cite{DZ} that
\[
\|\nabla_x u\|_{C^\beta((-4,0]\times\R^n)}+\|u_t\|_{C^\beta((-4,0]\times\R^n)}\le C (\|u\|_{L^\infty((-6,0])\times\R^n)}+\|f\|_{C^\beta((-6,0]\times\R^n)}),
\]
where $\beta=\min(\gamma,\alpha)$. If $\beta=\alpha$, then we are done. If $\beta<\alpha$, we can apply Theorem 1.1 in \cite{DZ} one more time to have
\[
\|\nabla_x u\|_{C^\alpha((-2,0]\times\R^n)}+\|u_t\|_{C^\alpha((-2,0]\times\R^n)}\le C (\|u\|_{L^\infty((-6,0])\times\R^n)}+\|f\|_{C^\alpha((-6,0]\times\R^n)}).
\]
\end{proof}
Let 
\[
Lu(t,x)=p.v.\int_{\R^n}(u(t,x+y)-u(t,x))K(t,x,y)dy.
\]
and
\begin{equation}\label{eq:g}
g(t,x)=\int_{\R^n}(u(t,x+y)-u(t,x))G(t,x,y)dy.
\end{equation}
The following calculations will be useful in proving Theorem \ref{thm:schauder-div} and Corollary \ref{cor:burgers}.
\begin{lem}\label{lem:aux1}
Let $K$ satisfy (K1), (K3) and  $|K(t,x,y)|\le \Lda |y|^{-n-1} \mbox{ for all }
  (t,x,y)\in (-6,0]\times\R^n\times\R^n$ (the lower bound in (K2) is not needed here). There exists $C>0$ depending only on $n,\Lambda$ such that
\[
\|Lu\|_{C^\alpha((-2,0]\times\R^n)}\le C\|u\|_{C^{1+\alpha}_{x,t}((-3,0]\times\R^n)}.
\]
\end{lem}
\begin{proof}
Since $K$ is symmetric in $y$, it is elementary to check that 
\[
\|Lu(t,\cdot)\|_{L^\infty(\R^n)}\le C\|u(t,\cdot)\|_{C^{1+\alpha}(\R^n)}.
\]

For $t\in (-2,0]$ and $x \in \R^n$, we have
\[
\begin{split}
&|Lu(t,x)-Lu(0,x)|\\
=&|\int_{\R^n}(u(t,x+y)-u(t,x))K(t,x,y)dy-\int_{\R^n}(u(0,x+y)-u(0,x))K(0,x,y)dy|\\
=&|\int_{B_{|t|}}(u(t,x+y)-u(t,x)-u(0,x+y)+u(0,x))K(0,x,y)dy|\\
&+|\int_{B^c_{|t|}}(u(t,x+y)-u(t,x)-u(0,x+y)+u(0,x))K(0,x,y)dy|\\
&+|\int_{\R^n}(u(t,x+y)-u(t,x))(K(t,x,y)-K(0,x,y)dy|\\[1ex]
\le & I+II + III.
\end{split}
\]
As far as term $I$ is concerned, we use the mean value theorem and (K1) in order to get
\begin{multline*}
\int_{B_{|t|}}(u(t,x+y)-u(t,x)-u(0,x+y)+u(0,x))K(0,x,y)dy  \\
 = 
\int_{B_{|t|}}(\nabla_x u(t,x+\theta y) - \nabla_x u(t,x) -\nabla_x u(0,x+\theta y)+\nabla_x u (0,x))\cdot y K(0,x,y)dy.
\end{multline*}
Since $|K(t,x,y)|\le \Lda |y|^{-n-1}$, we obtain
\begin{align*}
I & \le \Lambda \|\nabla_x u\|_{C^\alpha((-3,0]\times\R^n)}   \int_{B_{|t|}} |y|^{1+\alpha} |y|^{-n-1} d y  \\
& \le C \|\nabla_x u\|_{C^\alpha((-3,0]\times\R^n)} |t|^\alpha.
\end{align*}

 To estimate $II$, we observe that by mean value theorem and that $|K(t,x,y)|\le \Lda |y|^{-n-1}$,
\[
\begin{split}
II&\le \int_{B^c_{|t|}}|u_t(s,x+y)-u_t(s,x)||t||K(0,x,y)|dy\quad\mbox{for some } s\in (t,0)\\
&\le \Lambda \|u_t(s,\cdot)\|_{C^\alpha(\R^n)} |t|\int_{B^c_{|t|}}\min(|y|^\alpha,1) |y|^{-n-1} dy\\
&\le C  \|u\|_{C^{1+\alpha}((-3,0]\times\R^n)} |t|^\alpha \\[1ex]
\end{split}
\]

To estimate $III$, we proceed similarly by using  (K3) in order to get
\[
\begin{split}
III&\le C\|\nabla_x u\|_{C^\alpha((-3,0]\times\R^n)} \int_{B_{1}}|y|^{1+\alpha}|K(t,x,y)-K(0,x,y)|dy\\
&+ C\|u\|_{L^\infty((-3,0]\times\R^n)}\int_{B^c_{1}}|K(t,x,y)-K(0,x,y)|dy)\\
&\le C(\|\nabla_x u\|_{C^\alpha((-3,0]\times\R^n)} +\|u\|_{L^\infty((-3,0]\times\R^n)}) |t|^\alpha.\\
\end{split}
\]

Therefore,
\[
|Lu(t,x)-Lu(0,x)|\le C\|u\|_{C^{1+\alpha}_{x,t}((-3,0]\times\R^n)}|t|^\alpha.
\]

Similarly, we have 
\[
\begin{split}
&|Lu(t,x)-Lu(t,0)|\\
=&|\int_{\R^n}(u(t,x+y)-u(t,x)-u(t,y)+u(t,0))K(t,0,y)dy|\\
&+|\int_{\R^n}(u(t,x+y)-u(t,x))(K(t,x,y)-K(t,0,y)dy|\\[1ex]
\le & C\| u(t,\cdot)\|_{C^{1+\alpha}(\R^n)}|x|^\alpha,
\end{split}
\]
where we have used the symmetry of $K$ in $y$ as above, and Lemma 2.4 of \cite{JX2} to estimate the first term in the left hand side of the inequality.

Finally, the desired estimate follows from standard translation arguments.
\end{proof}

\begin{lem}\label{lem:aux2}
Let $\beta\in(0,\alpha)$,  $G$ satisfy (G1) and (G2), and $g$ be defined as in \eqref{eq:g}. There exists $C>0$ depending only on $n,\Lambda,\beta,\alpha$ such that
\[
\|g\|_{C^\beta((-4,0]\times\R^n)}\le C\|u\|_{Lip_{x,t}((-5,0]\times\R^n)}.
\]
\end{lem}
\begin{proof}
It is clear that
\[
\|g\|_{L^\infty((-4,0]\times\R^n)}\le C(\|\nabla_x u\|_{L^\infty((-4,0]\times\R^n)}+\|u\|_{L^\infty((-4,0]\times\R^n)}).
\]
Moreover, for $(t,x)\in(-4,0]\times \R^n$,  $s \in (-4,0]$ and $\beta<\alpha$, we have
\[
\begin{split}
|g(t,x)-g(s,x)|
\le &|\int_{\R^n}(u(t,y)-u(t,x))(G(t,x,y)-G(s,x,y)dy|\\
&+|\int_{\R^n}(u(t,y)-u(t,x)-u(s,y)+u(s,x))G(s,x,y)dy|\\
\le & C \|u\|_{Lip}\int_{\R^n}\min(|y|,1)\min(|t-s|^\alpha,|y|^\alpha)|y|^{-n-1}dy\\
&+C \|u\|_{Lip}\int_{\R^n}\min(|y|,|t-s|)\min(1,|y|^\alpha)|y|^{-n-1}dy\\
\le & C\|u\|_{Lip}|t-s|^\alpha|\log |t-s||\le C\|u\|_{Lip}|t-s|^\beta
\end{split}
\]
where 
\[ \|u\|_{Lip((-5,0] \times \R^n)} = \| u\|_{L^\infty ((-5,0] \times \R^n)}+ \| u_t \|_{L^\infty ((-5,0] \times \R^n)} + \| \nabla_x u \|_{L^\infty ((-5,0] \times \R^n)}.\]
Similarly, for $(t,x) \in (-4,0] \times \R^n$, $z \in \R^n$ and $\beta <\alpha$, we have 
\[
\begin{split}
&|g(t,x)-g(t,z)|\\
&\le|\int_{\R^n}(u(t,x+y)-u(t,x))(G(t,x,y)-G(t,z,y))dy|\\
&+|\int_{\R^n}(u(t,x+y)-u(t,x)-u(t,z+y)+u(t,z))G(t,z,y)dy|\\
&\le C \|u\|_{Lip}\int_{\R^n}\min(|y|,1)\min(|x-z|^\alpha,|y|^\alpha) |y|^{-n-1}dy\\
&+C \|u\|_{Lip}\int_{\R^n}\min(|y|,|x-z|)\min(1,|y|^\alpha) |y|^{-n-1} dy\\
&\le C\|u\|_{Lip}|x-z|^\alpha|\log |x-z||\le C\|u\|_{Lip}|x-z|^\beta.
\end{split}
\]
We conclude that for $\beta < \alpha$, 
\[ \| g \|_{C^\beta ((-4,0] \times \R^n)} \le C \|u\|_{Lip}.\]

\end{proof}

We shall also need the following iteration
lemma.
\begin{lem}[Lemma 1.1 in \cite{GG}]\label{lem:iteration}
  Let $h: [T_0,T_1] \to \R$ be nonnegative and bounded. Suppose that
  for all $0\le T_0\le t<s\le T_1$ we have
\[
h(t)\le A(s-t)^{-\gamma}+\frac12 h(s)
\]
with $\gamma>0$  and $A >0$. Then
there exists $C=C(\gamma)$ such that for all $T_0\le t <s \le T_1$ we have
\[
h(t)\le C A(s-t)^{-\gamma}. 
\]
\end{lem}

\section{Proofs of the main results}


\begin{proof}[Proof of Theorem \ref{thm:schauder-div}]
By Proposition~\ref{thm:schauder in xt} and Lemma \ref{lem:aux2}, 
\[
\begin{split}
\|u\|_{C^{1+\beta}((-2,0]\times\R^n)}&\le C (\|u\|_{L^\infty((-4,0])\times\R^n)}+\|g\|_{C^\beta((-4,0]\times\R^n)}+\|f\|_{C^\beta((-4,0]\times\R^n)})\\
&\le C (\|u\|_{Lip((-5,0] \times \R^n)}+\|f\|_{C^\beta((-5,0]\times\R^n)}).
\end{split}
\]
Let
\[
h(\gamma,s)= \begin{cases} 
[\nabla_x u]_{C^\gamma((s,0]\times\R^n)}+[u_t]_{C^\gamma((s,0]\times\R^n)} & \text{ if } \gamma \in (0,1) \\
\|\nabla_x u\|_{L^\infty((s,0]\times\R^n)}+\| u_t\|_{L^\infty((s,0]\times\R^n)} & \text{ if } \gamma = 0.
\end{cases}
\]
Then we just proved that
\begin{equation}\label{eq:interpolation}
h(\beta,-2)\le C (\|u\|_{L^\infty((-5,0])\times\R^n)}+\|f\|_{C^\beta((-5,0]\times\R^n)}+h(0,-5)).
\end{equation}
For every $-2<\tau<s\le -1$, if we let
\[
v(t,x)=u\left(\mu t+t_0,\mu x\right)\quad\mbox{with }\mu=\frac{s-\tau}{3},\ t_0=\frac{5s-2\tau}{3},
\]
then $v$ satisfies that 
\[
\begin{split}
v_t(t,x)&=\int_{\R^n}(v(t,x+y)-v(t,x))\tilde K(t,x,y)dy\\
&+\int_{\R^n}(v(t,x+y)-v(t,x))\tilde G(t,x,y)dy+\tilde f(t,x)\quad\mbox{in } (-6,0]\times\R^n,
\end{split}
\]
where
\[
\tilde K(t,x,y)=\mu^{n+1}K(\mu t+t_0,\mu x,\mu y),\ \tilde G(t,x,y)=\mu^{n+1}G(\mu t+t_0,\mu x,\mu y)
\]
and $\tilde f(t,x)=f(\mu t+t_0, \mu x)$. Since $\mu<1$, each of $\tilde K$, $\tilde G$ and $\tilde f$ satisfies the same assumptions on $K, G, f$, respectively. Therefore, \eqref{eq:interpolation} holds for $v$ as well.  Rescaling back to $v$, we have for $-2<\tau<s\le -\frac12$,
\[
h(\beta,s)\le \frac{C}{|\tau-s|^{1+\beta}}(\|u\|_{L^\infty((-5,0])\times\R^n)}+\|f\|_{C^\beta((-5,0]\times\R^n)})+\frac{C}{|\tau-s|^{\beta}}h(0,\tau).
\]
 By an interpolation inequality, we know that for every $\epsilon<1$, there exists $C>0$ independent of $\va$ such that 
\[
h(0,\tau) \le \epsilon h(\beta,\tau) +C\epsilon^{-\frac{1}{\beta}} \|u\|_{L^\infty((-6,0])\times\R^n)}.
\]
Choosing $\epsilon=\frac{|\tau-s|^{\beta}}{2C}$, we have that
\[
h(\beta,s)\le \frac{1}{2}h(\beta,\tau)+\frac{C}{|\tau-s|^{1+\beta}}(\|u\|_{L^\infty((-6,0])\times\R^n)}+\|f\|_{C^\beta((-5,0]\times\R^n)})).
\]
By the iteration lemma, Lemma \ref{lem:iteration}, we have that 
\[
h(\beta, -1)\le C (\|u\|_{L^\infty((-6,0])\times\R^n)}+\|f\|_{C^\beta((-6,0]\times\R^n)}).
\]
This proves Theorem \ref{thm:schauder-div}.
\end{proof}

\begin{proof}[Proof of Corollary \ref{cor:burgers}]
First of all, it follows from \cite{CCV} that a positive smooth periodic solution $w$ of \eqref{eq:main} satisfies
\[
\|w\|_{C^\alpha_{x,t}((-5,0]\times\R^n)}\le C
\]
for some $\alpha, C>0$ depending only on $n, \|w\|_{L^\infty((-6,0]\times\R^n)}$ and $\min_{(-6,0]\times\R^n}w$.

Now let $m$ be as in \eqref{eq:m}, and 
\[
K(t,x,y)=\frac{m(t,x,x)}{|y|^{n+1}} \quad\mbox{and} \quad G(t,x,y)=\frac{m(t,x,x+y)-m(t,x,x)}{|y|^{n+1}}.
\]
Then $\|m\|_{C^{\alpha}((-5,0]\times\R^n\times\R^n)}\le C$. Therefore, it is elementary to check that $K$ and $G$ satisfy the assumptions in Theorem \ref{thm:schauder-div}. Therefore, we have
\[
\|w\|_{C^{1+\beta_1}_{x,t}((-5,0]\times\R^n)}\le C(\beta_1)
\]
for all $\beta_1<\alpha$. 

Differentiating \eqref{eq:main} in $x$, we have for $v=\nabla_x w$,
\[
\begin{split}
\partial_t v&= p.v. \int_{\R^n}(v(t,x+y)-v(t,x))\frac{m(t,x,x)}{|y|^{n+1}}dy\\
&+p.v. \int_{\R^n}(v(t,x+y)-v(t,x))\frac{m(t,x,y+x)-m(t,x,x)}{|y|^{n+1}}dy\\
&+ p.v. \int_{\R^n}(w(t,x+y)-w(t,x))\frac{2\nabla_x m(t,x,x)}{|y|^{n+1}}dy\\
&+ p.v. \int_{\R^n}(w(t,x+y)-w(t,x))\frac{2(\nabla_x m(t,x,y+x)-\nabla_x m(t,x,x))}{|y|^{n+1}}dy.\\
&=I+II+III+IV.
\end{split}
\]
Here, for simplicity of the writing, we used that $m(t,x,y)=m(t,y,x)$, but actually this is not needed. The proof will go through without using this symmetry. 

Notice that $\|\nabla_x m(t,x,y+x)\|_{C^{\beta_1}_{x,t}((-5,0]\times\R^n\times\R^n)}\le C(\beta_1)$. It follows from Lemma \ref{lem:aux1} that
\[
\|III\|_{C^{\beta_1}_{x,t}((-5,0]\times\R^n)}\le \|w\|_{C^{1+\beta_1}_{x,t}((-5,0]\times\R^n)}\le C(\beta_1).
\]
It follows from Lemma \ref{lem:aux2} that
\[
\|IV\|_{C^{\beta_2}_{x,t}((-5,0]\times\R^n)}\le C(\beta_1,\beta_2)
\]
for every $\beta_2<\beta_1$. Applying Theorem \ref{thm:schauder-div}, we obtain
\[
\|v\|_{C^{1+\beta_2}_{x,t}((-5,0]\times\R^n)}\le C(\beta_1,\beta_2).
\]
That is,
\[
\|\nabla w\|_{C^{1+\beta_2}_{x,t}((-5,0]\times\R^n)}\le C(\beta_1,\beta_2).
\]
Similarly, we can differentiate \eqref{eq:main} in $t$ and obtain
\[
\|w_t\|_{C^{1+\beta_2}_{x,t}((-5,0]\times\R^n)}\le C(\beta_1,\beta_2).
\]
Then this corollary follows from keeping differentiating \eqref{eq:main} and applying Theorem \ref{thm:schauder-div} as above.
\end{proof}

Let us now address the question of global well-posedness of periodic solutions to \eqref{burgers} with positive initial data. Let $u_0$ be, say, $2\pi$-periodic with bounds $0 < m_0< u_0(x)< M_0<\infty$. Note that positive solutions to \eqref{burgers} enjoy both the maximum and minimum principles. So, the initial bounds hold a priori for all time. This preserves the  uniform parabolicity of \eqref{eq:main3}. Based on close similarity of \eqref{burgers} to the Euler equation, \cite{ISV} develops a parallel classical local well-posedness theory for \eqref{burgers} with initial condition in $H^s$ with $s>n/2+1+\e$, along with the analogue of the Beale-Kato-Majda blowup criterion. Thus, for a mollified initial data $u_\d$, $\d>0$, we have a local time interval of existence $I_\d$ enjoying the same uniform $L^\infty$ bounds from above and below. By the symmetrization \eqref{eq:m}, this solution gains $C^\a$-regularity for some $\a>0$ with bounds $\|u_\d(t)\|_{C^\a(\R^n)} \leq C t^{-\a} \|u_0\|_{L^\infty}$, for $t\in I_\d$, where $C>0$ is independent of $\d$. Similar bounds hold for the kernel $m$. Our Theorem~\ref{thm:schauder-div} now applies to provide $C^{1+\beta}$ regularity, which in particular by rescaling in time reads $\|u_\d(t)\|_{C^{\beta+1}(\R^n)} \leq C t^{-\beta-1} \|u_0\|_{L^\infty}$. Thus the BKM criterion clearly holds at the end of the interval $I_\d$, and hence $u_\d$ can be extended beyond $I_\d$ and to infinity since the bounds improve. We can now pass to the limit as $\d \to 0$ on any finite interval to obtain global weak solution starting from $L^\infty$-data, which in turn becomes $C^\infty$ instantaneously due to Corollary~\ref{cor:burgers}.

We also note that the long-time dynamics of solutions to \eqref{burgers} is described by exponentially fast convergence to a constant state, the state that is consistent with the energy conservation law for \eqref{burgers}. This dual nature of the equation, conservative and dissipative, ensures presence of the strong inverse energy cascade similar to the one observed in turbulence theory of a two dimensional fluid.

Finally we remark, that although for our purposes the assumption of $C^{1+\a}$ regularity on $w$ in Theorem~\ref{thm:schauder-div} was sufficient, the theorem still holds under any other assumption $C^{1+\e}$ unrelated to $\a$ that appears in (K1)--(G2). It is only necessary to make sense of the integral expressions classically.


\bigskip
\bigskip

\noindent C. Imbert

\noindent Department of Mathematics and Applications, CNRS \& \'Ecole Normale Sup\'erieure (Paris) \\ 
45 rue d'Ulm, 75005 Paris, France. \\ [1mm]
Email: \textsf{Cyril.Imbert@ens.fr}

\bigskip

\noindent T. Jin

\noindent Department of Mathematics, The Hong Kong University of Science and Technology\\
Clear Water Bay, Kowloon, Hong Kong

\smallskip
and
\smallskip

\noindent Department of Computing and Mathematical Sciences, California Institute of Technology \\
1200 E. California Blvd., MS 305-16, Pasadena, CA 91125, USA\\[1mm]
Email: \textsf{tianlingjin@ust.hk} / \textsf{tianling@caltech.edu}

\bigskip

\noindent R. Shvydkoy

\noindent Department of Mathematics, Statistics, and Computer Science,  M/C 249, \\
University of Illinois, Chicago, IL 60607, USA\\[1mm]
Email: \textsf{shvydkoy@uic.edu}

\end{document}